\documentclass[12pt]{amsart}
\usepackage{amsmath,amsfonts,amssymb,amscd,amsthm,amsbsy,epsf}
\newtheorem{thm}{Theorem}[section]

\newtheorem{cor}[thm]{Corollary}

\theoremstyle{definition} 		
\newtheorem{defn}[thm]{Definition}
\newtheorem{remark}[thm]{Remark}
\newtheorem{remarks}[thm]{Remarks}
\newtheorem*{note}{Notation}

\def\nat{{\mathbb N}}
\def\zat{{\mathbb Z}}
\def\qat{{\mathbb Q}}

\def\B{{\mathcal B}}

\def\F{{\mathcal F}}

\def\Tau{{\mathfrak T}}

\begin{document}
\title{Density theorems for rational numbers}
\author{Andreas Koutsogiannis}

\begin{abstract}
Introducing the notion of a rational system of measure preserving transformations and proving a recurrence result for such systems, we give sufficient conditions in order a subset of rational numbers to contain arbitrary long arithmetic progressions.
\end{abstract}

\keywords{IP-systems, density-type theorems, rational numbers}
\subjclass{Primary 11R45; 28D05}

\maketitle
\baselineskip=18pt
\pagestyle{plain}


\section*{Introduction}

In 1927, van der Waerden proved (in \cite{vdW}) that for any finite partition of the set of natural numbers, there exists a cell of the partition which contains arbitrary long arithmetic progressions, which is a (perhaps the most) fundamental result of Ramsey theory. The density version of the van der Waerden theorem, that any set of positive upper density in $\nat$ possesses arbitrary long arithmetic progressions (the \textit{upper density} of a subset $A\subseteq \nat$ is defined by $\overline{d}(A)=\limsup_n \frac{|A\cap\{1,\ldots,n\}|}{n},$ where $|A|$ denotes the cardinality of $A$) was conjectured by Erd$\ddot{\text{o}}$s and Tur\'{a}n in 1930's and established by Szemer\'{e}di in 1975 (\cite{Sz}).

Furstenberg, in 1977, reproved Szemer\'{e}di's theorem, introducing a correspondence principle, which provides the link between density Ramsey theory and ergodic theory (\cite{Fu}) and proving a multiple recurrence theorem for measure preserving systems.

In this paper we will prove a multiple recurrence and two density results for the set of rational numbers giving sufficient conditions in order a set of rational numbers to contain arbitrary long arithmetic progressions. Using a representation of rational numbers (proved in \cite{BIP}), according to which every rational number can be represented as a dominated located word over an infinite alphabet, we define the notion of a rational system (Definition~\ref{rmps}). We obtain:

(a) a multiple recurrence result concerning rational systems of measure preserving transformations in Theorem~\ref{thm:2}, using the analogous result of Furstenberg-Katznelson for the IP-systems of measure preserving transformations,

(b) a sufficient condition via F{\o}lner sequences in order a subset of rational numbers to contain arbitrary long arithmetic progressions in Theorem~\ref{thm:4}, using a result of Hindman and Strauss for infinite countable left cancellative semigroups; and

(c) a density result viewing the rational numbers as located words in Theorem~\ref{l}, which follows from the density Hales-Jewett theorem of Furstenberg-Katznelson.

\begin{note}
Let $\nat=\{1,2,\ldots\}$ be the set of natural numbers, $\mathbb{Z}=\{\ldots,-2,-1,0,1,2,\ldots\}$ be the set of integer numbers,  $\qat=\{\frac{m}{n}: m\in \zat,\;n\in \nat\}$ the set of rational numbers and $\zat^\ast=\zat\setminus\{0\},$ $\qat^\ast=\qat\setminus\{0\}.$
\end{note}

\section{A multiple recurrence result for rational numbers}

In this section, using the representation of rational numbers as dominated located words over an infinite alphabet, we define (in Definition~\ref{rmps}) the rational systems and we prove, in Theorem~\ref{thm:2} below, a recurrence result for such systems for measure preserving transformations using an analogous result for IP-systems of Furstenberg-Katznelson.

According to \cite{BIP}, every rational number $q$  has a unique expression in the form $$q=\sum^{\infty}_{s=1}q_{-s}\frac{(-1)^{s}}{(s+1)!}\;+\;\sum^{\infty}_{r=1}q_{r}(-1)^{r+1}r! $$ where $(q_n)_{n \in \mathbb{Z}^\ast}\subseteq \nat\cup\{0\}$ with $\;0\leq q_{-s}\leq s$ for every $s>0$, $ 0\leq q_r\leq r$ for every $r> 0$ and $q_{-s}=q_r=0$ for all but finite many $r,s$.

So, for a non-zero rational number $q$, there exist unique $l\in \nat,$ a non-empty finite subset of non-zero integers, the \textit{domain} of $q,$ $\{t_1<\ldots<t_l\}=dom(q)\in [\zat^\ast]^{<\omega}_{>0}$ and a subset of natural numbers $\{q_{t_1},\ldots,q_{t_l}\}\subseteq \nat$ with $1\leq q_{t_i}\leq -t_i$ if $t_i<0$ and $1\leq q_{t_i}\leq t_{i}$ if $t_i> 0$ for every $1\leq i\leq l,$ such that if $dom^-(q)=\{t\in dom(q):\;t<0\}$ and $dom^+(q)=\{t\in dom(q):\;t>0\}$ to have $$ q=\sum_{t\in dom^-(q)}q_t\frac{(-1)^{-t}}{(-t+1)!}\;+\;\sum_{t\in dom^+(q)}q_t(-1)^{t+1}t!\;\;\;(\text{we set}\;\;\sum_{t\in \emptyset}=0).$$ Consequently, $q$ can be represented as the word \begin{center} $q=q_{t_1}\ldots q_{t_l}.$ \end{center} It is easy to see that $$e^{-1}-1=-\sum^{\infty}_{t=1}\frac{2t-1}{(2t)!}< \sum_{t\in dom^-(q)}q_t\frac{(-1)^{-t}}{(-t+1)!} <\sum^{\infty}_{t=1}\frac{2t}{(2t+1)!}=e^{-1}$$ and that $$ \sum_{t\in dom^+(q)}q_t(-1)^t(t+1)! \in \zat^\ast\;\;\text{if}\;\;dom^+(q)\neq\emptyset.$$

\medskip

We will now recall  the notion of IP-systems introduced by Furstenberg and Katznelson in \cite{FuKa}.

\begin{defn}
Let $\{T_n\}_{n\in \nat}$ be a set of commuting transformations of a space. To every multi-index $\alpha=\{t_1,\ldots,t_l\}\in [\nat]^{<\omega}_{>0}, \;t_1<\ldots<t_l,$ we attach the transformation \begin{center} $\mathcal{T}_{\alpha}=T_{t_1}\ldots T_{t_l}.$ \end{center} The corresponding family $\{\mathcal{T}_{\alpha}\}_{\alpha\in [\nat]^{<\omega}_{>0}}$ is an \textit{IP-system} (of transformations).

 Two IP-systems $\{\mathcal{T}^{(1)}_{\alpha}\}_{\alpha\in [\nat]^{<\omega}_{>0}},$ $\{\mathcal{T}^{(2)}_{\alpha}\}_{\alpha\in [\nat]^{<\omega}_{>0}}$ defined by $\{T^{(1)}_{n}\}_{n\in \nat}$ and $\{T^{(2)}_{n}\}_{n\in \nat}$ respectively, are \textit{commuting} if $T^{(1)}_{n}T^{(2)}_{m}=T^{(2)}_{m}T^{(1)}_{n}$ for every $n,$ $m\in \nat.$
\end{defn}

\begin{thm}[Furstenberg-Katznelson, \cite{FuKa}]\label{thm:1}
Let $\{\mathcal{T}^{(1)}_{\alpha}\}_{\alpha\in [\nat]^{<\omega}_{>0}},\ldots,\{\mathcal{T}^{(k)}_{\alpha}\}_{\alpha\in [\nat]^{<\omega}_{>0}}$ be $k$ commuting IP-systems defined by the measure preserving transformations $\{T^{(j)}_{n}\}_{n\in \nat},$ $1\leq j\leq k$ of a measure space $(X,\B,\mu)$ with $\mu(X)=1$ (i.e. $T_n^{(j)}$ is $\B$-measurable with $\mu(T_n^{(j)-1}(A))=\mu(A)$ for every $A\in\B,$ $1\leq j\leq k,$ $n\in \nat$). If $A\in \B$ with $\mu(A)>0,$ then there exists an index $\alpha$ with
\begin{center} $\mu(A\cap \mathcal{T}^{(1)-1}_{\alpha}(A)\cap\ldots \cap \mathcal{T}^{(k)-1}_{\alpha}(A))>0.$ \end{center}
\end{thm}

We will define the notion of a rational system (see also \cite{K}), extending the notion of an IP-system.

\begin{defn}\label{rmps} Let $\{T_n\}_{n\in \zat^{\ast}}$ be a set of commuting transformations of a set $X.$ For a non-zero rational number $q$ represented as the word $q=q_{t_1}\ldots q_{t_l},$ we define \begin{center} $\mathcal{T}^q(x)=T_{t_1}^{q_{t_1}}\ldots T_{t_l}^{q_{t_l}}(x)$ and $\mathcal{T}^0(x)=x$ for every $x\in X.$\end{center}  The corresponding family $\{\mathcal{T}^{q}\}_{q\in\qat}$ is a \textit{rational system} (of transformations).

Two rational systems $\{\mathcal{T}_1^{q}\}_{q\in\qat},$ $\{\mathcal{T}_2^{q}\}_{q\in\qat}$ defined by $\{T_{n,1}\}_{n\in \zat^{\ast}}$ and $\{T_{n,2}\}_{n\in \zat^{\ast}}$ respectively, are \textit{commuting} if $T_{n,1}T_{m,2}=T_{m,2}T_{n,1}$ for every $n,$ $m\in \zat^{\ast}.$
\end{defn}

Using Theorem~\ref{thm:1}, we can take the following:

\begin{thm}\label{thm:2}
Let $\{\mathcal{T}_1^{q}\}_{q\in\qat},\ldots,\{\mathcal{T}_k^{q}\}_{q\in\qat}$ be $k$ commuting rational systems defined by the measure preserving transformations $\{T_{n,j}\}_{n\in \zat^{\ast}},$ $1\leq j\leq k$ respectively of a measure space $(X,\B,\mu)$ with $\mu(X)=1.$ If $A\in \B$ and $\mu(A)>0$ then there exists $q\in \qat^\ast$ with
\begin{center} $\mu(A\cap (\mathcal{T}_1^q)^{-1}(A)\cap\ldots \cap (\mathcal{T}_k^q)^{-1}(A))>0.$ \end{center}
\end{thm}

\begin{proof}
Let $A\in \B$ with $\mu(A)>0.$  For every $t\in \zat^\ast$ choose a natural number $q_t$ with $1\leq q_t\leq |t|.$ For every $t\in \nat,$ $1\leq i\leq k,$ set $\phi^{(i)}_{2t-1}=T_{t,i}^{q_t},$ $\phi^{(i)}_{2t}=T_{-t,i}^{q_{-t}}$ and let the corresponding IP-systems $\{\phi^{(i)}_{\alpha}\}_{\alpha\in [\nat]^{<\omega}_{>0}}.$

According to Theorem~\ref{thm:1} there exists $\alpha=\{t_1<\ldots<t_l\}\in [\nat]^{<\omega}_{>0}$ with \begin{center} $\mu(A\cap \phi^{(1)-1}_{\alpha}(A)\cap\ldots \cap \phi^{(k)-1}_{\alpha}(A))>0.$\end{center} Since $\phi^{(i)}_{\alpha}=\mathcal{T}_i^q$ for all $1\leq i\leq k,$ where $$q=\sum_{2t\in \alpha}q_t\frac{(-1)^{-t}}{(-t+1)!}+\sum_{2t-1\in \alpha}q_t (-1)^{t+1}t!\in \qat^\ast,$$ we have that $\mu(A\cap (\mathcal{T}_1^q)^{-1}(A)\cap\ldots \cap (\mathcal{T}_k^q)^{-1}(A))>0.$
\end{proof}

\begin{remark}\label{r}
With the same arguments as in Theorem~\ref{thm:2}, we can prove that:

$(1)$ There exists $q\in \zat^\ast$ which satisfies the conclusion of Theorem~\ref{thm:2}, setting $\phi^{(i)}_{t}=T_{t,i}^{q_t}$ for every $t\in \nat$ and $1\leq i\leq k$.

$(2)$ There exists $q\in (e^{-1}-1,e^{-1})\cap \qat^{\ast}$ which satisfies the conclusion of Theorem~\ref{thm:2}, setting $\phi^{(i)}_{t}=T_{-t,i}^{q_{-t}}$ for every $t\in \nat$ and $1\leq i\leq k$.
\end{remark}

\section{Density conditions for rational numbers}

In this section, using Theorem~\ref{thm:2}, we will give, via left F{\o}lner sequences, a sufficient condition (in Theorem~\ref{thm:4}) in order a subset of rational numbers to contain arbitrary long arithmetic progressions, using a result of Hindman and Strauss (Theorem~\ref{thm:33}). Also, using Furstenberg-Katznelson's density Hales-Jewett theorem for words over a finite alphabet (Theorem~\ref{thm:6}), we prove in Theorem~\ref{l} a density result viewing the rational numbers as located words.

Firstly, we will define the left F{\o}lner sequences.

\begin{defn}\label{lfs}
Let $(S,+)$ be a semigroup. A \textit{left F{\o}lner sequence} in $[S]^{<\omega}_{>0}$ is a sequence $\{F_n\}_{n\in \nat}$ in $[S]^{<\omega}_{>0}$ such that for each $s\in S,$ $$\lim_{n\rightarrow\infty}\frac{|(s+F_n)\triangle F_n|}{|F_n|}=0,$$ where $A\triangle B=(A\setminus B)\cup (B\setminus A).$
\end{defn}

\begin{remark}[\cite{HS}]
If $(S,+)$ is an infinite countable left cancellative semigroup (i.e. $a+b=a+c\;\Rightarrow\;b=c$ for every $a,b,c\in S$), then we can find a left F{\o}lner sequence in $[S]^{<\omega}_{>0}.$
\end{remark}

Given a left F{\o}lner sequence $\F=\{F_n\}_{n\in \nat}$ in $[S]^{<\omega}_{>0},$ there is a natural notion of upper density associated with $\F,$ namely $$\overline{d}_{\F}(A)=\lim\sup_{n\rightarrow\infty}\frac{|A\cap F_n|}{|F_n|}.$$

\medskip

In order to prove Theorem~\ref{thm:4}, which gives a sufficient condition in order a subset of rational numbers to contain arbitrary long arithmetic progressions, we will need some notions from the theory of ultrafilters and also Theorem~\ref{thm:33}, a fundamental result of Hindman and Strauss, which we mention below.

\subsection*{Ultrafilters}
Let $X$ be a non-empty set. An \textit{ultrafilter} on the set $X$ is a zero-one finite additive measure $\mu$ defined on all the subsets of $X$. The set of all ultrafilters on the set $X$ is denoted by $\beta X.$ So, $\mu\in\beta X$ if and only if
\begin{itemize}
\item[{(i)}] $\mu(A)\in\{0,1\}$ for every $A\subseteq X$, $\mu(X)=1$, and
\item[{(ii)}] $\mu(A\cup B)=\mu(A)+\mu(B)$ for every $A,B\subseteq X$ with $A\cap B=\emptyset$.
\end{itemize}
For $\mu\in\beta X,$ it is easy to see that $\mu(A\cap B)=1,$ if $\mu(A)=1$ and $\mu(B)=1$. For every $x\in X$ is defined the \textit{principal ultrafilter} $\mu_x$ on $X$ which corresponds a set $A\subseteq X$ to $\mu_x(A)=1$ if $x\in A$ and  $\mu_x(A)=0$ if $x\notin A$.  So, $\mu$ is a non-principal ultrafilter on $X$ if and only if  $\mu(A)=0$ for every finite subset $A$ of $X$.

The set $\beta X$ becomes a compact Hausdorff space if it be endowed with the topology $\Tau$ which has basis the family $\{\overline{A} :  A\subseteq X \}$, where $\overline{A}=\{\mu\in\beta X : \mu(A)=1 \}$. It is easy to see that $\overline{A\cap B}=\overline{A}\cap \overline{B}$, $\overline{A\cup B}=\overline{A}\cup \overline{B}$ and $\overline{X\setminus A}=\beta X\setminus \overline{A}$ for every $A,B\subseteq X$. We always consider the set $\beta X$ endowed with the topology $\Tau.$ Also $\beta X$ is called the \textit{Stone-$\check{\text{C}}$ech compactification} of the set $X.$

If $(X,+)$ is a semigroup, then a binary operation $+$ is defined on $\beta X,$ extending the operation $+$ on $X,$ corresponding to every $\mu_1, \mu_2\in \beta X$ the ultrafilter $\mu_1 + \mu_2\in \beta X,$ with
\begin{center}
$(\mu_1 + \mu_2)(A)= \mu_1(\{x\in X : \mu_2(\{y\in X : x+y\in A\})=1\})$ for every $A\subseteq X$.
\end{center}
With this operation $(\beta X,+,\Tau)$ becomes a right topological semigroup, that is, for every $\mu\in\beta X$ the function $f_{x_0} : \beta X\rightarrow \beta X$ with $f_{x_0}(\mu)=\mu_{x_0}+\mu$ is continuous.

\medskip

Hindman and Strauss in \cite{HS} proved the following result concerning left cancellative semigroups. For $A\subseteq S$ and $t\in S$ we set $-t+A=\{s\in S:\;t+s\in A\}.$

\begin{thm}\label{thm:33}
Let $S$ be an infinite countable left cancellative semigroup, let $\F=\{F_n\}_{n\in \nat}$ be a left F${\o}$lner sequence in $[S]^{<\omega}_{>0},$ and let $A\subseteq S.$ There is a countably additive measure $\mu$ on the set $\B$ of Borel subsets of $\beta S$ such that

$(1)$ $\mu(\overline{A})=\overline{d}_{\F}(A),$

$(2)$ for all $B\subseteq S, $ $\mu(\overline{B})\leq \overline{d}_{\F}(B),$

$(3)$ for all $B\in \B$ and all $t\in S,$ $\mu(-t+B)=\mu(B)=\mu(t+B),$ and

$(4)$ $\mu(\beta S)=1.$
\end{thm}

Now we can state our theorem and give a proof in analogy to Theorem 5.6 in \cite{HS}.

\begin{thm}\label{thm:4}
Let $\F=\{F_n\}_{n\in \nat}$ be a left F{\o}lner sequence in $[\qat]^{<\omega}_{>0},$ and let $A\subseteq \qat$ such that $\overline{d}_{\F}(A)>0.$ Then for each $k\in \nat$ there exists $q\in \qat^\ast$ such that \begin{center} $\overline{d}_{\F}(A\cap(-q+A)\cap\ldots\cap (-kq+A))>0.$ \end{center}
\end{thm}

\begin{proof}
Let $\B$ be the set of Borel subsets of $\beta \qat.$ Pick a countably additive measure $\mu$ on $\B$ which satisfies the conditions of Theorem~\ref{thm:33}. Let $k\in \nat.$ For $l\in \{1,\ldots,k\}$ and $\nu\in \beta \qat$ let $T_{l,m}(\nu)=\mu_{l(-1)^{m+1}m!}+\nu$ for $m\in \nat$ and $T_{l,m}(\nu)=\mu_{l\frac{(-1)^{-m}}{(-m+1)!}}+\nu$ for $-m\in \nat.$ Each $T_{l,n},$ $1\leq l\leq k,\;n\in \zat^\ast$ is a continuous function from $\beta\qat$ to $\beta\qat,$ since $\beta\qat$ is a right topological semigroup. Let $\{\mathcal{T}_l^{q}\}_{q\in\qat}$ be the rational system defined from $\{T_{l,n}\}_{n\in \zat^\ast}$ for every $1\leq l\leq k$ respectively.

The transformations $T_{l,n},$ $1\leq l\leq k,\;n\in \zat^\ast$ preserve $\mu,$ since $\mu$ satisfies the condition (3) of Theorem~\ref{thm:33}.  Consequently, $\{\mathcal{T}_l^{q}\}_{q\in\qat},$ $1\leq l\leq k$ are commuting rational systems of measure preserving transformations on the space $(\beta \qat,\B,\mu).$ According to the condition (1) of Theorem~\ref{thm:33}, we have $\mu(\overline{A})=\overline{d}_{\F}(A)>0.$ So, using Theorem~\ref{thm:2}, we can find $q\in \qat^\ast$ such that \begin{center} $\mu(\overline{A}\cap (\mathcal{T}_1^q)^{-1}(\overline{A})\cap\ldots \cap (\mathcal{T}_k^q)^{-1}(\overline{A}))>0.$ \end{center} This gives \begin{center} $\overline{d}_{\F}(A\cap(-q+A)\cap\ldots\cap (-kq+A))=\mu(\overline{A\cap(-q+A)\cap\ldots\cap (-kq+A)})=$\\ $ \;\;\;\;\;\;\;\;\;\; \;\;\;\;\;\;\;\;\;\; \;\;\;\;\;\;\;\;\;\; \;\;\;\;\;\;\;\;\;\; \;\;\;\;\;\;\;\;\;\; \;\;\;\;\;\;\;\;\;\;\;\;=\mu(\overline{A}\cap (\mathcal{T}_1^q)^{-1}(\overline{A})\cap\ldots \cap (\mathcal{T}_k^q)^{-1}(\overline{A}))>0,$ \end{center} which finishes the proof.
\end{proof}

\begin{cor}\label{c}
Let $\F=\{F_n\}_{n\in \nat}$ be a left F{\o}lner sequence in $[\qat]^{<\omega}_{>0},$ and let $A\subseteq \qat$ such that $\overline{d}_{\F}(A)>0.$ Then for each $k\in \nat$ there exist $q\in \qat^\ast$ and $p\in A$ such that \begin{center} $p+jq\in A$ for every $0\leq j\leq k.$ \end{center}
\end{cor}

\begin{proof}
Let $k\in \nat.$ According to the proof of Theorem~\ref{thm:4} there exists $q\in \qat^\ast$ such that $\overline{d}_{\F}(A\cap(-q+A)\cap\ldots\cap (-kq+A))>0.$ Pick $p\in A\cap(-q+A)\cap\ldots\cap (-kq+A).$
\end{proof}

\begin{remarks} $(1)$ In the statements of Theorem~\ref{thm:4} and Corollary~\ref{c} the rational number $q$ can be located either in $\zat^\ast$ or in $(e^{-1}-1,e^{-1})\cap\qat^{\ast},$ using in the respective proof the results of Remark~\ref{r}.

$(2)$ Defining for a left F{\o}lner sequence $\F=\{F_n\}_{n\in \nat}$ in $[\qat]^{<\omega}_{>0},$   and $A\subseteq \qat$  the density \begin{center} $d^{\ast}_{\F}(A)=\sup\big\{\alpha:\;(\forall m\in \nat)(\exists n\geq m)(\exists q\in \qat)(|A\cap (q+F_n)|\geq \alpha |F_n|)\big\}$ \end{center} we can replace $\overline{d}_{\F}$ with $d^{\ast}_{\F}$ in the statements of Theorem~\ref{thm:4} and Corollary~\ref{c}, using  Theorem 4.6 in \cite{HS}, instead of Theorem~\ref{thm:33}.
\end{remarks}

Viewing the rational numbers as words and using the density Hales-Jewett theorem of Furstenberg and Katznelson (\cite{FuKa}), we will prove in Theorem~\ref{l} another density result for the set of rational numbers. Let start with the necessary notation.

Let $\Sigma=\{\alpha_1,\ldots,\alpha_k\}$ for $k\in \nat$ a finite set and $\upsilon\notin\Sigma.$ We denote by $W(\Sigma)$ the set of all the words $w=w_1\ldots w_n,$ where $n\in \nat$ and $w(i)\in \Sigma$ for every $1\leq i\leq n,$ and by $W(\Sigma,\upsilon)$ the set of all the (variable) words in $W(\Sigma\cup\{\upsilon\})$ with at least one occurrence of the symbol $\upsilon.$ A \textit{combinatorial line} in $W(\Sigma)$ is a set $\{w(\alpha):\;\alpha\in \Sigma\}$ obtained by substituting the variable $\upsilon$  of the variable word $w(\upsilon)$ by the symbols $\alpha_1,\ldots,\alpha_k.$ We also denote by $W_n(\Sigma)$ the subset of $W(\Sigma)$ consisting of all the words of length $n.$

Furstenberg and Katznelson in \cite{FuKa2} proved the following theorem:

\begin{thm}\label{thm:6}
Let $\Sigma=\{\alpha_1,\ldots,\alpha_k\},\;k\in \nat$ a finite alphabet. If $A\subseteq W(\Sigma)$ and $\limsup_n \frac{|A\cap W_n(\Sigma)|}{k^{n}}>0,$  then $A$ contains a combinatorial line.
\end{thm}

\noindent  For every $n,\;k\in \nat$ and integers $t^n_1<\ldots<t_n^n$ with $|t^n_j|\geq k$ for every $1\leq j\leq n$ we define, in analogy to the representation of rational numbers, the subset $\qat(t^n_1,\ldots,t_n^n,k)\subseteq \qat^{\ast}$ as   \begin{center}$\qat(t^n_1,\ldots,t_n^n,k)=\{\sum^{n}_{j=1}q_{t_j^n}c_{t_j^n}:$ $1\leq q_{t_j^n}\leq k,$ $c_{t_j^n}=\frac{(-1)^{-t_j^n}}{(-t_j^n+1)!}$ if $t_j^n<0$\\ $\;\;\;\;\;$ and  $c_{t_j^n}=(-1)^{t_j^n+1}t_j^n!$ if $t_j^n>0 \}.$\end{center}

\medskip

For $\Sigma=\{1,\ldots,k\},\;k\in \nat$ we define  \begin{center} $g:W(\Sigma)\rightarrow \bigcup_{n\in \nat} \qat(t^n_1,\ldots,t_n^n,k)$ with $g(w_1\ldots w_n)=\sum^{n}_{j=1}w_jc_{t_j^n}.$ \end{center} Note that $g|_{W_n(\Sigma)}:W_n(\Sigma)\rightarrow \qat(t^n_1,\ldots,t_n^n,k)$ is $1-1$ and onto for every $n\in \nat.$

Using Theorem~\ref{thm:6} we have the following density theorem:

\begin{thm}\label{l}
Let $k\in \nat$ and a sequence $((t^n_1,\ldots,t_n^n))_{n\in \nat}$ with $t^n_1<\ldots<t_n^n$ and $|t^n_j|\geq k$ for every $1\leq j\leq n,$ $n\in \nat.$ If $A\subseteq \qat$ with $\limsup_{n}\frac{|A\cap \qat(t^n_1,\ldots,t_n^n,k)|}{k^{n}}>0,$ then there exist $p\in A$ and $q\in \qat^\ast$ with $dom(p),$ $dom(q)\subseteq \{t^n_1,\ldots,t_n^n\}$ for some $n\in \nat,$ such that \begin{center} $p+iq\in A$ for every $i=0,1,\ldots,k-1.$\end{center}
\end{thm}

\begin{proof}
Let $\Sigma=\{1,\ldots,k\}.$ Since $|g^{-1}(A)\cap W_n(\Sigma)|=|A\cap \qat(t^n_1,\ldots,t_n^n,k)|$ for every $n\in \nat,$ the set $g^{-1}(A)$ contains a combinatorial line $\{w(\alpha):\;\alpha\in \Sigma\}$ obtained by a variable word $w(\upsilon),$ $\upsilon\notin \Sigma,$ according to Theorem~\ref{thm:6}. Let $n$ be the length of $w(\upsilon).$ Then, $\{g(w(\alpha)):\;\alpha\in \Sigma\}\subseteq A\cap \qat(t^n_1,\ldots,t_n^n,k).$ So, there exist $F_1=\{t\in dom(w):w_t\in \Sigma\},$ $F_2=\{t\in dom(w):w_t=\upsilon\}$ with $F_1,\;F_2\subseteq \{t^n_1,\ldots,t_n^n\}$ and $F_1\cap F_2=\emptyset$ such that if  $q=\sum_{t\in F_2}c_t$ and $p=q+\sum_{t\in F_1}w_tc_t,$ where $c_{t}=\frac{(-1)^{-t}}{(-t+1)!}$ if $t<0$ and  $c_{t}=(-1)^{t+1}t!$ if $t>0,$  we have that $g(w(i))=p+iq\in A$ for every $0\leq i\leq k-1.$
\end{proof}

\medskip

\section*{Acknowledgments}

The author wish to thank Professor V. Farmaki for helpful discussions and support during the preparation of this paper.

\bigskip
{\footnotesize

\noindent
\newline
Andreas Koutsogiannis:
\newline
{\sc Department of Mathematics, Athens University, Panepistemiopolis, 15784 Athens, Greece}
\newline
E-mail address: akoutsos@math.uoa.gr}
\end{document}